\documentclass[a4paper, reqno,12pt]{amsart}

\usepackage{tikz, tikz-3dplot}
\usepackage{amsmath, amsthm, amssymb,graphics }

\theoremstyle{plain}
\newtheorem{te}{Theorem}
\newtheorem{lem}[te]{Lemma}

\newtheorem{pr}[te]{Proposition}

\newtheorem{con}[te]{Conjecture}
\theoremstyle{remark}

\newtheorem*{ack*}{Acknowledgment}

\def\t{{\theta}}

\def\R{{\mathbb R}}

\def\C{{\mathbb C}}
\def\S{{\mathbb S}}

\def\A{{\mathbb A}}

\def\nint{\mathop{\diagup\kern-13.0pt\int}}

\def\Cc{{\mathcal C}}\def\Nc{{\mathcal N}}

\def\Pc{{\mathcal P}}

\def\emph#1{{\it #1}}

\begin{document}
		\author{Jean Bourgain}
		\address{School of Mathematics, Institute for Advanced Study, Princeton NJ}
		\email{bourgain@math.ias.edu}
		\author{Ciprian Demeter}
		\address{Department of Mathematics, Indiana University,  Bloomington IN}
		\email{demeterc@indiana.edu}
		\author{Dominique Kemp}
		\address{Department of Mathematics, Indiana University,  Bloomington IN}
		\email{dekemp@umail.iu.edu}
		\thanks{The first two authors are partially supported by the Collaborative Research NSF grant DMS-1800305.  }

		\title[Decouplings for real analytic surfaces of revolution]{Decouplings for real analytic surfaces of revolution}\maketitle

\begin{abstract}
We extend the decoupling results of the first two authors to the case of real analytic surfaces of revolution in $\R^3$. New examples of interest include the torus and the perturbed  cone.	
\end{abstract}
\maketitle

\section{Background and the main result}

Let $$S=\{(\xi_1,\xi_2,g(\xi_1,\xi_2)):\;(\xi_1,\xi_2)\in [-1,1]^2\}$$
be a smooth, compact  surface in $\R^3$, given by the graph of the function $g$. For each $0<\delta<1$ let $\Nc_\delta(S)$ be the $\delta$-neighborhood of $S$.

Given a function $f:\R^3\to\C$ and a set $\tau\subset \R^3$, we denote by $f_\tau$ the Fourier restriction of $f$ to $\tau$.

In \cite{BD3}, \cite{BD4}, the first two authors proved the following result.
\begin{te}
\label{4}	
Assume $S$ has everywhere nonzero Gaussian curvature.  Let $\Pc_\delta(S)$ be a partition of $\Nc_\delta(S)$ into near rectangular boxes $\tau$ of dimensions $\sim \delta^{1/2}\times \delta^{1/2} \times \delta$. Then for each $f$ Fourier supported in $\Nc_\delta(S)$ and for $2\le p\le 4$ we have
\begin{equation}
\label{1}
\|f\|_{L^p(\R^3)}\lesssim_{\epsilon}(\delta^{-1})^{\frac12-\frac1p+\epsilon}(\sum_{\tau\in\Pc_\delta(S)}\|f_\tau\|^p_{L^p(\R^3)})^{1/p}.
\end{equation}
Moreover, if Gaussian curvature is positive then
\begin{equation}
\label{2}
\|f\|_{L^p(\R^3)}\lesssim_{\epsilon}\delta^{-\epsilon}(\sum_{\tau\in\Pc_\delta(S)}\|f_\tau\|^2_{L^p(\R^3)})^{1/2}.
\end{equation}
\end{te}
Inequality \eqref{2} is referred to as an $l^2$- decoupling.
It is false for $p>4$.

Inequality \eqref{1} is an $l^p$-decoupling. Since there are roughly $\delta^{-1}$ boxes in $\Pc_\delta(S)$, the $l^p$-decoupling follows from the $l^2$-decoupling and H\"older's inequality when $S$ has positive curvature. However, if $S$ has negative curvature, the stronger $l^2$-decoupling may fail. This is easiest to observe in the case of the hyperbolic paraboloid, corresponding to $g(\xi_1,\xi_2)=\xi_1^2-\xi_2^2.$ What rules out the $l^2$-decoupling here is the fact that this surface contains at least one line, and the following elementary principle (applied with $N\sim \delta^{-1/2}$).
\begin{pr}
\label{5}
Let $L$ be a line segment in $\R^n$ of length $\sim 1$. For each $0 \le \delta, N^{-1}<1$, let $\Pc_{\delta,N}$ be a partition of the $\delta$-neighborhood $\Nc_\delta(L)$ of $L$ into $\sim N$ cylinders $T$  with length $N^{-1}$ and radius $\delta$.

For $p>2$ let $D(\delta,N,p)$ be the smallest constant such that
\begin{equation}
\label{3}
\|f\|_{L^p(\R^n)}\le D(\delta,N,p)(\sum_{T\in\Pc_{\delta,N}}\|f_T\|_{L^p(\R^n)}^2)^{1/2}
\end{equation}
holds for all $f$ with Fourier transform supported on $\Nc_\delta(L)$. Then
$$D(\delta,N,p)\sim N^{\frac12-\frac{1}{p}},$$
and (approximate) equality in \eqref{3} can be achieved by using a smooth approximation of $1_{\Nc_\delta(L)}$.
\end{pr}

The implicit constants in \eqref{1} and \eqref{2} depend on $\epsilon$, on the $C^3$ norm of $g$ and on the lower bound for the Gaussian curvature. In \cite{BD3} and \cite{BD4}, inequalities \eqref{2} and \eqref{1} are first proved for the model surfaces, the elliptic and hyperbolic paraboloid, respectively. The extension to the more general surfaces in Theorem \ref{4}  is then obtained via local approximation and induction on scales, using Taylor's formula with cubic error term. This is the reason why the third derivatives are also important, in addition to the first and second order ones.
\medskip

The notable feature of the choice of the diameter $\delta^{1/2}$ of each  $\tau\in\Pc_\delta(S)$ in Theorem \ref{4} is that this is the largest scale for which $\tau$ can be thought of  as being essentially flat. By that we mean that there is a rectangular box $R_\tau$ such that $R_\tau\subset \tau\subset 1000R_\tau$.
This is of course a consequence of the nonzero curvature condition. The case when one of the principal curvatures is zero leads to new types of decoupling, that have been only partially explored (see also the last section). For future reference, we record the result from \cite{BD3} for the cone
$$\Cc^2:=\{(\xi_1,\xi_2,\sqrt{\xi_1^2+\xi_2^2}\;):\;\frac14\le \xi_1^2+\xi_2^2\le 4\}$$
and the cylinder
$$\Cc yl^2:=\{(\xi_1,\xi_2,\xi_3):\;\xi_1^2+\xi_2^2=1,\;|\xi_3|\lesssim  1\}.$$
\begin{te}
\label{6}	
For $S$ either $\Cc^2$ or $\Cc yl^2$ we let $\Pc_\delta(S)$ be a partition of $\Nc_\delta(S)$ into roughly $\delta^{-1/2}$ essentially  rectangular plates $P$ with dimensions $\sim 1\times \delta^{1/2}\times \delta.$ Then for each $2\le p\le 6$ and each $f$ with Fourier transform supported in  $\Nc_\delta(S)$ we have
$$\|f\|_{L^p(\R^3)}\lesssim_\epsilon\delta^{-\epsilon}(\sum_{P\in\Pc_{\delta}(S)}\|f_P\|_{L^p(\R^3)}^2)^{1/2}.$$
\end{te}
The fact that we decouple using plates of length $\sim 1$  is enforced by Proposition \ref{5}. The range $[2,6]$ here is larger than the range $[2,4]$ from Theorem \ref{4} because of subtle dimensionality considerations.

As an immediate corollary of H\"older's inequality, we get the following $l^4$ decoupling for $S=\Cc^2,\Cc yl^2$, analogous to \eqref{1}
\begin{equation}
 \label{8}
 \|f\|_{L^4(\R^3)}\lesssim_\epsilon\delta^{-\epsilon-\frac1{8}}(\sum_{P\in\Pc_{\delta}(S)}\|f_P\|_{L^4(\R^3)}^4)^{1/4}.
 \end{equation}

We will refer to this inequality for the cylinder  as {\em cylindrical decoupling}.

\bigskip

A natural step would be to try  to extend Theorems \ref{4} and \ref{6} to the case of arbitrary real analytic surfaces $S$ in $\R^3$, without any restriction on curvature.
One of the issues is identifying the correct dimensions of the boxes in the partition of $\Pc_\delta(S)$. In analogy to the previous examples, we would like these boxes to be essentially flat. One possible way to formalize the question is recorded in the following conjecture.
\begin{con}
\label{9}
If $S$ is the graph of a nonconstant real analytic function $g:[-1,1]^2\to\R$ then for each $0<\delta\le 1$ there is a partition $\Pc_\delta(S)$ of $\Nc_\delta(S)$ into essentially flat boxes $\tau$ (of possibly different dimensions) such that for each $f$ with Fourier transform supported in $\Nc_\delta(S)$ we have
$$
\|f\|_{L^4(\R^3)}\lesssim_{\epsilon}\delta^{-\epsilon}|\Pc_{\delta}(S)|^{\frac14}(\sum_{\tau\in\Pc_\delta(S)}\|f_\tau\|^4_{L^4(\R^3)})^{1/4},$$
where $|\Pc_{\delta}(S)|$ refers to the cardinality of $\Pc_{\delta}(S)$.
\end{con}

In this generality, identifying such a partition seems to be a rather difficult task. We will limit our investigation to the class of surfaces of revolution, which as we shall soon see, is large enough to include some interesting new examples.
\smallskip

To get started, for each real analytic function $\gamma:[\frac12,2]\to\R$ we consider the associated surface of revolution
$$S_\gamma=\{(\xi_1,\xi_2,\gamma(\sqrt{\xi_1^2+\xi_2^2}\;)):\;\frac14\le \xi_1^2+\xi_2^2\le 4\}.$$
For example, the cone $\Cc^2$ corresponds to $\gamma(r)=r$.
Our main result can be somewhat vaguely summarized as follows. We save the details about the precise definition of $\Pc_\delta(S)$ for the later sections. The interesting new feature of the partitions $\Pc_\delta(S)$ is that they will consist of boxes of different scales. 
\begin{te}[Main result]
\label{11}	
Conjecture \ref{9} holds for all real analytic surfaces of revolution $S_\gamma$.
\end{te}

As we shall soon see, the curvature of $S_\gamma$ is zero exactly when either $\gamma'$ or $\gamma''$ is zero.
Let $r_1,\ldots,r_M$ be the zeros of $\gamma'\gamma''$ inside $[\frac12,2]$.
The fact that there are only finitely many such zeros is a consequence of the real analyticity of $\gamma$.
We consider pairwise disjoint intervals $I_i=(r_i-\Delta_i,r_i+\Delta_i)$, with $\Delta_i$ small enough  such that the power series expansion of $\gamma$ centered at $r_i$ has radius of convergence $>\Delta_i$. Various other restrictions on the smallness of $\Delta_i$ will become apparent throughout the forthcoming argument. Note that the complement $$[\frac12,2]\setminus\bigcup_{i=1}^MI_i=\bigcup J_i$$ is the union of at most $M+1$ intervals $J_i$.
The triangle inequality will allow us to separately consider the part of the surface corresponding to one such interval. On the  intervals $J_i$ the surface will have nonzero curvature, so Theorem \ref{4} is applicable.

It remains to investigate the contribution from the intervals $I_i$. Let us fix such an  interval.
To simplify notation, we will assume it to be $(1-\Delta,1+\Delta)$.

The  partition $\Pc_\delta(S)$ and the type of analysis we will employ will depend on the derivatives of $\gamma$ at 1. These derivatives encode all the necessary information concerning the size of the two principal curvatures of $S_\gamma$. This will be explored in more detail the next section.

\section{A case analysis based on principal curvatures}

Differential geometry ties the notion of curvature of surfaces $S$ in $\R^3$ to the change in the direction of the normal vector along curves in $S$. To be exact, it describes curvature by way of the derivative of the map $N: S \to \S^2$, whose value at $p$ is the unit (outward) normal vector of $S$ at $p$.

When $S$ is given as the graph of a function $g$, this differential in local coordinates $(\xi_1, \xi_2)$ has the form
\begin{equation} (1+(g_1)^2+(g_2)^2)^{-\frac{3}{2}}\begin{pmatrix} g_{11}(1+(g_2)^2)-g_1g_2g_{12} & & &  g_{12}(1+(g_2)^2)-g_1g_2g_{22} \\ \\ \\  g_{12}(1+(g_1)^2) - g_1g_2g_{11} & & &  g_{22}(1+(g_1)^2)-g_1g_2g_{12} \end{pmatrix} \label{M}\end{equation} where $g_i = \frac{\partial g}{\partial \xi_i}$ and $g_{ij} = \frac{\partial^2g}{\partial \xi_i \partial \xi_j}$.

With a little algebra, the determinant  (also known as the Gaussian curvature of $S$) at a point $(\xi_1, \xi_2, g(\xi_1, \xi_2))$ is found to be
\begin{equation}
\label{22}
K_S(\xi_1, \xi_2) = \frac{g_{11}g_{22} - (g_{12})^2}{(1+(g_1)^2+(g_2)^2)^{2}}.
\end{equation}
The two eigenvalues $\lambda_1,\lambda_2$ are called principal curvatures. Their product equals the Gaussian curvature.

For a later convenience, we record the simplified version of \eqref{22} in the case that $S$ is the surface of revolution  $S_\gamma$. The Gaussian curvature along $\sqrt{(\xi_1)^2+(\xi_2)^2} = r$ is
\begin{equation}\label{10}   K(r) = \frac{\gamma'(r)\gamma''(r)}{r(1+\gamma'(r)^2)^2}.
\end{equation}
To motivate our intuition in the following sections, we also record the following well known formulae for the principal curvatures in the radial and angular directions
$$|\lambda_{rad}(r)|=\frac{|\gamma''(r)|}{(1+(\gamma'(r))^2)^{3/2}}$$
$$|\lambda_{ang}(r)|=\frac{|\gamma'(r)|}{r(1+(\gamma'(r))^2)^{1/2}}.$$

We will split our analysis into three cases.

\bigskip

\textbf{Case 1.}  If $\gamma'(1)\not=0$ and $\gamma^{(n)}(1)=0$ for all $n\ge 2$, then we have in fact $\gamma(r)=\gamma'(1)r$. This is a cone, so it is covered by Theorem \ref{6}.
The next two cases are new.
\bigskip

\textbf{Case 2.}  If $\gamma'(1)=\ldots=\gamma^{(n-1)}(1)=0$ and $\gamma^{(n)}(1)\not=0$ for some $n\ge 2$, then the angular principal curvature is zero along the curve $r=1$.  We will refer to these manifolds as {\em quasi-tori} and will discuss them in Section \ref{s3}.

The typical example to have in mind is the torus, corresponding to
\begin{equation}
\label{24}
\gamma(r)=(\frac14-(r-1)^2)^{1/2}
\end{equation}
defined on $(1-\Delta,1+\Delta)$, $\Delta<\frac12$.
\bigskip

\textbf{Case 3.}  If $\gamma'(1)\not=0$, $\gamma''(1)=\ldots=\gamma^{(n-1)}(1)=0$ and $\gamma^{(n)}(1)\not=0$ for some $n\ge 3$, then the radial principal curvature is zero along the curve $r=1$. These manifolds can be thought of as perturbations of the cone and will be discussed in Section \ref{s4}.

\section{The case of the quasi-torus}
\label{s3}

To simplify notation we will assume $\gamma(1)=1$ and $\gamma^{(n)}(1)=n!$, so that
\begin{equation}
\label{12}
\gamma(r)=1+(r-1)^{n}+O((r-1)^{n+1}).
\end{equation}
Fix $\delta$. Our task is to describe the partition $\Pc_\delta(S)$.  Recall that we want each element of $\Pc_\delta(S)$ to be an essentially rectangular box.
\medskip

We start with a dyadic decomposition near 1
$$[1-\Delta,1+\Delta]=[1-\delta^{1/n},1+\delta^{1/n}]\cup \cup_{k\ge 1} \{r:\;|r-1|\in(2^{k-1}\delta^{1/n},2^{k}\delta^{1/n}]\,\}.$$
Note that $k$ is restricted to  $O(\log \frac1\delta)$ values. Thus, since we can afford $\epsilon$ losses in Theorem \ref{11} we may invoke again the triangle inequality and restrict our attention to a fixed $k$. Due to symmetry, we may further restrict attention to the right halves of these sets in the above decomposition, which we call $U_k$.

For $k\ge 0$ let us call $S_k$ the part of the surface $S_\gamma$ above the thin annulus
$$\A_k=\{(\xi_1,\xi_2):\;(\xi_1^2+\xi_2^2)^{1/2}\in U_k\}.$$
Figure 1 depicts $S_0,S_1,S_2$, with $S_0$ being the nearly horizontal circular strip at the top. The rationale for bringing in such a decomposition is that the two principal curvatures are essentially constant on each $S_k$
$$|\lambda_{rad}(r)|\sim (2^k\delta^{1/n})^{n-2}$$
$$|\lambda_{ang}(r)|\sim (2^k\delta^{1/n})^{n-1}.$$

\bigskip

We will  first see how to deal with the surface $S_0$ corresponding to the interval $U_0=[1,1+\delta^{1/n}]$.   Note that  $\Nc_\delta(S_0)$ sits inside the $C\delta^{1/n}$-neighborhood  of the cylinder $\Cc yl^2$, with $C=O(1)$.
We may thus apply cylindrical decoupling  (Theorem \ref{6}) with $\delta$ replaced with $\delta^{1/n}$. Each vertical plate of dimensions $\sim 1\times \delta^{\frac1{2n}}\times \delta^{\frac1n}$ will intersect $S_0$ in a cap $\theta_0$ with dimensions $\sim  \delta^{1/2n}\times \delta^{1/n}$. Note that for each such $\theta_0$, the box $\Nc_\delta(\theta_0)$ is essentially flat.
\medskip

Let $\Pc_\delta(S_0)$ be the partition consisting of all boxes $\tau_0=\Nc_\delta(\theta_0)$. Invoking cylindrical decoupling, we find that whenever $f$ has Fourier transform supported inside $\Nc_\delta(S_0)$ we have

$$
\|f\|_{L^4(\R^3)}\lesssim_{\epsilon}\delta^{-\epsilon}|\Pc_{\delta}(S_0)|^{\frac14}(\sum_{\tau_0\in\Pc_\delta(S_0)}\|f_{\tau_0}\|^4_{L^4(\R^3)})^{1/4}.$$

The collection $\Pc_\delta(S_0)$ will provide the first elements of the final partition $\Pc_\delta(S)$.

\bigskip

\textbf{Figure 1.} The partition $\Pc_\delta(S)$

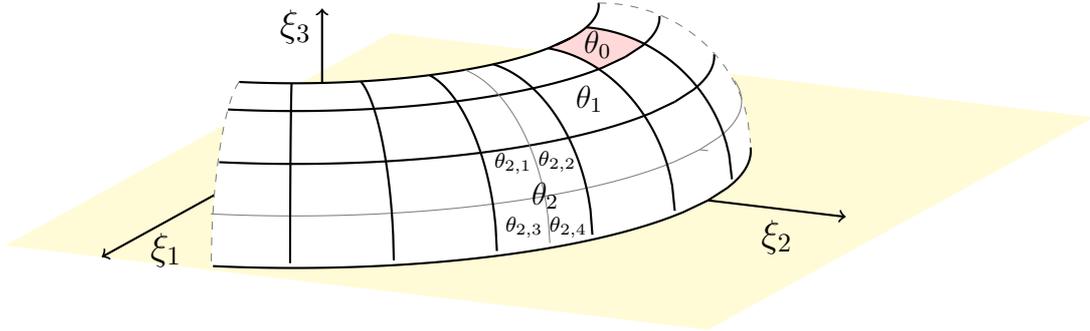
	\begin{center}
		\tdplotsetmaincoords{75}{25}
		\begin{tikzpicture}[tdplot_main_coords, scale=4]
		
		\begin{scope}
		
		
		\fill[fill=yellow!20] (-0.35,-1.5, 0) --  (-0.35, 1.5, 0) --  (2.2, 1.5, 0) --  (2.2,-1.5, 0) -- cycle;

		
		\draw[black, thick, ->]
		(-1/4,0,0) -- (2,0,0)
		;
		\draw[black, thick, <-]
		(0,-1.5,0) -- (0,1/2,0)
		;
		\draw[black, thick, ->]
		(1/10,0,-1/6) -- (1/10,0,1/2)
		;
		
		
		\fill[fill=white]
		({0.2}, {sqrt(1 - 0.2*0.2}, {1/2})
		\foreach \t in {0.2, 0.21, ..., 1}
		{
			--({\t}, {sqrt(1- \t*\t)}, {1/2})
		}
		\foreach \t in {1, 1.01, ..., 1.5}
		{
			--({\t}, {0}, {sqrt(1/4 - (\t-1)*(\t-1))})
		}
		\foreach \t in {1.5, 1.49, ..., 1}
		{
			--({\t/5}, {\t*sqrt(24)/5}, {sqrt(1/4 - (\t-1)*(\t-1))})
		}
		;

		\fill[fill=white]
		({0.5}, {sqrt(3)/2}, {1/2})
		\foreach \t in {1, 1.01, ..., 1.5}
		{
			--({\t/2}, {\t*sqrt(3)/2}, {sqrt(1/4 - (\t-1)*(\t-1))})
		}
		\foreach \t in {0.75, 0.76, ..., 1.5}
		{
			--({\t}, {sqrt(1.5*1.5 - \t*\t)}, {0})
		}
		;
		
		\fill[fill=white]
		({0.75}, {sqrt(7)/4}, {1/2})
		\foreach \t in {1, 1.01, ..., 1.5}
		{
			--({\t*3/4}, {\t*sqrt(7)/4}, {sqrt(1/4 - (\t-1)*(\t-1))})
		}
		\foreach \t in {1.125, 1.135, ..., 1.5}
		{
			--({\t}, {sqrt(1.5*1.5 - \t*\t)}, {0})
		}
		;
		
		\fill[fill=white]
		({0.9}, {sqrt(19)/10}, {1/2})
		\foreach \t in {1, 1.01, ..., 1.5}
		{
			--({\t*9/10}, {\t*sqrt(19)/10}, {sqrt(1/4 - (\t-1)*(\t-1))})
		}
		\foreach \t in {1.35, 1.36, ..., 1.5}
		{
			--({\t}, {sqrt(1.5*1.5 - \t*\t)}, {0})
		}
		;
		
		\fill[fill=white]
		({1}, {0}, {sqrt(1/4)})
		\foreach \t in {1, 0.99, ..., 0.25}
		{
			--({\t}, {-sqrt(1-\t*\t)}, {sqrt(1/4})
		}
		\foreach \t in {1, 1.01, ..., 1.5}
		{
			--({\t/4}, {-sqrt(15)*\t/4}, {sqrt(1/4 - (\t-1)*(\t-1))})
		}
		\foreach \t in {0.375, 0.385, ..., 1.5}
		{
			--({\t}, {-sqrt(1.5*1.5 - \t*\t)}, {0})
		}
		\foreach \t in {1.5, 1.49, ..., 1}
		{
			--({\t}, {0}, {sqrt(1/4 - (\t-1)*(\t-1))})
		}
		;
		
		
		\fill[fill=red! 15]
		({sqrt(1/(1+0.15*0.15))}, {0.15*sqrt(1/(1+0.15*0.15))}, {1/2})
		\foreach \t in {0.15, 0.14, ..., -0.15}
		{
			--({sqrt(1 - \t*\t)}, {\t}, {sqrt(1/4))})
		}
		\foreach \t in {1, 1.01, ..., 1.2}
		{
			--({\t*sqrt(1/(1+0.15*0.15))}, {-0.15*\t*sqrt(1/(1+0.15*0.15))}, {sqrt(1/4 - (\t - 1)*(\t-1))})
		}
		\foreach \t in {-0.2, -0.19, ..., 0.2}
		{
			--({sqrt(1.2*1.2-\t*\t)}, {\t}, {sqrt(1/4 - 0.2*0.2})
		}
		\foreach \t in {1.2, 1.19, ..., 1}
		{
			--({\t*sqrt(1/(1+0.15*0.15))}, {0.15*\t*sqrt(1/(1+0.15*0.15))}, {sqrt(1/4 - (\t - 1)*(\t-1))})
		}
		
		;
		

		\draw[black, thick]
		({1}, {0}, {1/2})
		\foreach \t in {1, 0.99, ..., 0.25}
		{
			--({\t}, {-sqrt(1- \t*\t)}, {1/2})
		}
		;

		\draw[black, thick]
		({sqrt(1 - 0.45*0.45)}, {0.45}, {1/2})
		\foreach \t in {0.45, 0.44, ..., -0.1}
		{
			--({sqrt(1 - \t*\t)}, {\t}, {1/2})
		}
		;

		
		\draw[black, thick]
		({1.2}, {0}, {sqrt(1/4 - 0.2*0.2)})
		\foreach \t in {1.2, 1.19, ..., 1.08}
		{
			--({\t}, {sqrt(1.2*1.2 - \t*\t)}, {sqrt(1/4 - 0.2*0.2)})
		}
		;
		
		\draw[black, thick]
		({1.2}, {0}, {sqrt(1/4 - 0.2*0.2)})
		\foreach \t in {1.2, 1.19, ..., 0.3}
		{
			--({\t}, {-sqrt(1.2*1.2 - \t*\t)}, {sqrt(1/4 - 0.2*0.2)})
		}
		;
		
		\draw[black, thick]
		({1.38}, {0}, {sqrt(1/4 - 0.38*0.38)})
		\foreach \t in {1.38, 1.37, ..., 1.23}
		{
			--({\t}, {sqrt(1.38*1.38 - \t*\t)}, {sqrt(1/4 - 0.38*0.38)})
		}
		;

		\draw[black, thick]
		({1.38}, {0}, {sqrt(1/4 - 0.38*0.38)})
		\foreach \t in {1.38, 1.37, ..., 0.34}
		{
			--({\t}, {-sqrt(1.38*1.38 - \t*\t)}, {sqrt(1/4 - 0.38*0.38)})
		}
		;
		
		\draw[gray, thin]
		({1.47}, {0}, {sqrt(1/4 - 0.47*0.47)})
		\foreach \t in {1.47, 1.46, ..., 1.20}
		{
			--({\t}, {sqrt(1.47*1.47 - \t*\t)}, {sqrt(1/4 - 0.47*0.47)})
		}
		;
		
		\draw[gray, thin]
		({1.5}, {0}, {sqrt(1/4 - 0.47*0.47)})
		\foreach \t in {1.47, 1.46, ..., 0.36}
		{
			--({\t}, {-sqrt(1.47*1.47 - \t*\t)}, {sqrt(1/4 - 0.47*0.47)})
		}
		;

		
		\draw[black, thick]
		({1.5}, {0}, {sqrt(1/4 - 0.5*0.5)})
		\foreach \t in {1.5, 1.49, ..., 1.34}
		{
			--({\t}, {sqrt(1.5*1.5 - \t*\t)}, {sqrt(1/4 - 0.5*0.5)})
		}
		;

		\draw[black, thick]
		({1.5}, {0}, {sqrt(1/4 - 0.5*0.5)})
		\foreach \t in {1.5, 1.49, ..., 0.375}
		{
			--({\t}, {-sqrt(1.5*1.5 - \t*\t)}, {sqrt(1/4 - 0.5*0.5)})
		}
		;

		
		\draw[gray, dashed]
		({sqrt(1/(1+0.5*0.5))}, {0.5*sqrt(1/(1+0.5*0.5))}, {1/2})
		\foreach \t in {1, 1.01, ..., 1.5}
		{
			--({\t*sqrt(1/(1+0.5*0.5))}, {0.5*\t*sqrt(1/(1+0.5*0.5))}, {sqrt(1/4 - (\t - 1)*(\t-1))})
		}
		;
		
		\draw[black, thick]
		({sqrt(1/(1+0.15*0.15))}, {0.15*sqrt(1/(1+0.15*0.15))}, {1/2})
		\foreach \t in {1, 1.01, ..., 1.5}
		{
			--({\t*sqrt(1/(1+0.15*0.15))}, {0.15*\t*sqrt(1/(1+0.15*0.15))}, {sqrt(1/4 - (\t - 1)*(\t-1))})
		}
		;

		\draw[black, thick]
		({sqrt(1/(1+0.15*0.15))}, {-0.15*sqrt(1/(1+0.15*0.15))}, {1/2})
		\foreach \t in {1, 1.01, ..., 1.5}
		{
			--({\t*sqrt(1/(1+0.15*0.15))}, {-0.15*\t*sqrt(1/(1+0.15*0.15))}, {sqrt(1/4 - (\t - 1)*(\t-1))})
		}
		;

		\draw[black, thick]
		({sqrt(1/(1+0.45*0.45))}, {-0.45*sqrt(1/(1+0.45*0.45))}, {1/2})
		\foreach \t in {1, 1.01, ..., 1.5}
		{
			--({\t*sqrt(1/(1+0.45*0.45))}, {-0.45*\t*sqrt(1/(1+0.45*0.45))}, {sqrt(1/4 - (\t - 1)*(\t-1))})
		}
		;
		
		
		\draw[gray, thin]
		({sqrt(1/(1+0.6*0.6))}, {-0.6*sqrt(1/(1+0.6*0.6))}, {1/2})
		\foreach \t in {1, 1.01, ..., 1.5}
		{
			--({\t*sqrt(1/(1+0.6*0.6))}, {-0.6*\t*sqrt(1/(1+0.6*0.6))}, {sqrt(1/4 - (\t - 1)*(\t-1))})
		}
		;

		\draw[black, thick]
		({sqrt(1/(1+0.8*0.8))}, {-0.8*sqrt(1/(1+0.8*0.8))}, {1/2})
		\foreach \t in {1, 1.01, ..., 1.5}
		{
			--({\t*sqrt(1/(1+0.8*0.8))}, {-0.8*\t*sqrt(1/(1+0.8*0.8))}, {sqrt(1/4 - (\t - 1)*(\t-1))})
		}
		;
		
		\draw[black, thick]
		({sqrt(1/(1+1.3*1.3))}, {-1.3*sqrt(1/(1+1.3*1.3))}, {1/2})
		\foreach \t in {1, 1.01, ..., 1.5}
		{
			--({\t*sqrt(1/(1+1.3*1.3))}, {-1.3*\t*sqrt(1/(1+1.3*1.3)}, {sqrt(1/4 - (\t - 1)*(\t-1))})
		}
		;
		
		\draw[black, thick]
		({sqrt(1/(1+2.2*2.2))}, {-2.2*sqrt(1/(1+2.2*2.2))}, {1/2})
		\foreach \t in {1, 1.01, ..., 1.5}
		{
			--({\t*sqrt(1/(1+2.2*2.2))}, {-2.2*\t*sqrt(1/(1+2.2*2.2)}, {sqrt(1/4 - (\t - 1)*(\t-1))})
		}
		;
		
		\draw[gray, dashed]
		({0.25}, {-sqrt(15)/4}, {1/2})
		\foreach \t in {1, 1.01, ..., 1.5}
		{
			--({\t/4}, {-sqrt(15)*\t/4}, {sqrt(1/4 - (\t-1)*(\t-1))})
		}
		;

		\node[below] at  ({0}, {-1}, {0}) {\large $\xi_1$};		
		\node[below] at  ({1.75}, {0}, {0}) {\large $\xi_2$};	
		\node[left] at  ({1/10}, {0}, {7/16}) {\large $\xi_3$};
		\node at ({1.1}, {0}, {sqrt(1/4 - 0.1*0.1)}) { $\theta_0$};
		\node at ({1.07}, {0}, {sqrt(1/4 - 0.4*0.4)}) { $\theta_1$};
		\node at ({0.94}, {0}, {sqrt(1/4 - 0.495*0.495)}) {\tiny{ $\theta_{2,2}$}};
		\node at ({0.78}, {0}, {sqrt(1/4 - 0.498*0.498)}) { \tiny{ $\theta_{2,1}$}};
		\node at ({0.98}, {0}, {-0.15}) {\tiny{ $\theta_{2,4}$}};
		\node at ({0.82}, {0}, {-0.17}) {\tiny{ $\theta_{2,3}$}};
		\node at ({0.91}, {0}, {-0.05}) {$\theta_{2}$};
		\end{scope}
		\end{tikzpicture}
	\end{center}
\bigskip

Let us now investigate $S_k$, for  $k\ge 1$. Fix  $f$  with Fourier transform  supported inside $\Nc_\delta(S_k)$. There will be two steps needed in order to produce the desired partition $\Pc_\delta(S_k)$. The first step is very similar to the one we did for $k=0$. Namely, we invoke cylindrical decoupling to write

\begin{equation}
\label{14}
\|f\|_{L^4(\R^3)}\lesssim_{\epsilon}\delta^{-\epsilon}|\tilde{\Pc}_{\delta}(S_k)|^{\frac14}(\sum_{\tau_k\in\tilde{\Pc}_\delta(S_k)}\|f_{\tau_k}\|^4_{L^4(\R^3)})^{1/4}.
\end{equation}
Each $\tau_k\in\tilde{\Pc}_\delta(S_k)$ is equal to  $\Nc_\delta(\theta_k)$ for some cap $\theta_k$ of dimensions $\sim (2^{k}\delta^{1/n})^{1/2}\times(2^{k}\delta^{1/n})$.
\bigskip

It is not hard to see that $\tau_k$ is curved. This is because $S_k$ has big radial curvature. More concretely, note that \eqref{12} forces the $\delta$-neighborhood of the graph of $\gamma$ on $U_k$ to be a curved tube. This observation suggests that each $f_{\tau_k}$ can be further decoupled into smaller pieces. The principal curvatures of $\theta_k$ while nonzero, are very small. Consequently, Theorem \ref{4} is not directly applicable. What compensates for the small curvatures is the fact that $\theta_k$ has tiny area. This will allow us to stretch it into a surface of scale $\sim 1$, whose principal curvatures are also $\sim 1$. To execute this strategy  we use a linear transformation in the style of parabolic rescaling.
\bigskip

To simplify notation, let us denote by $s_k$ the scale $2^{k}\delta^{1/n}$. It is also convenient to deal with $\theta_k$ sitting directly above the $\xi_2$ axis, so that a point $(\xi_1,\xi_2,\xi_3)\in\theta_{k}$ satisfies
$$|\xi_1|\lesssim s_k^{1/2},\;\;\;\;\xi_2-1\sim s_k,\;\;\;\;\xi_3-1\sim s_k^n.$$
 We will use the transformation
$$L_k(\xi_1,\xi_2,\xi_3)=(\frac{\xi_1}{s_k^{1/2}},\frac{\xi_2-1}{s_k},\frac{\xi_3-1}{s_k^n}).$$
Let us call $\theta_{k,new}=L_k(\theta_k)$. We make a few observations related to this new surface.
\medskip

First, note that $L_k(\Nc_\delta(\theta_k))\subset \Nc_{\frac{\delta}{ s_k^{n}}}(\theta_{k,new})$. Thus the function $f_{new}$ defined by $$\widehat{f_{new}}=\widehat{f_{\tau_k}}\circ L_k^{-1}$$
has Fourier transform supported in $\Nc_{\frac{\delta}{ s_k^{n}}}(\theta_{k,new})$.

Second, note that the equation of $\theta_{k,new}$ in the new coordinates $\eta_1,\eta_2,\eta_3$ is
$$\eta_3=\frac{\gamma(\sqrt{1+s_k(\eta_1^2+2\eta_2)+s_k^2\eta_2^2}\,)-1}{s_k^n},\;\;|\eta_1|\lesssim 1,\;\;\eta_2\sim 1.$$
Using \eqref{12} and the fact that $\sqrt{1+r}=1+\frac{r}{2}+O(r^2)$ we may write
\begin{align*}
\eta_3&=\frac{1}{2^ns_k^n}(s_k(\eta_1^2+2\eta_2)+s_k^2\eta_2^2)^{n}+O(s_k^{-n}(s_k(\eta_1^2+2\eta_2)+s_k^2\eta_2^2)^{n+1})\\&= \frac{1}{2^n}(\eta_1^2+2\eta_2)^n+O(s_k)\Psi(\eta_1,\eta_2).
\end{align*}
Here $\Psi$ is a $C^\infty$ function.
Let $S_{ref}$ be the surface
$$\{(\eta_1,\eta_2,\frac{1}{2^n}(\eta_1^2+2\eta_2)^n),\;\;|\eta_1|\lesssim 1,\;\;\eta_2\sim 1\}.$$
The fact that $n\ge 2$ and the discussion from the previous section implies that $S_{ref}$ has both principal curvatures $\sim 1$.  The same remains true for $\theta_{k,new}$, as $s_k\ll 1$.

We can thus apply Theorem \ref{4} to decouple $f_{new}$ using $N$ essentially flat  boxes  $B$ of dimensions $\sim (\frac{\delta}{s_k^n})^{1/2}\times(\frac{\delta}{s_k^n})^{1/2} \times(\frac{\delta}{s_k^n})$
\begin{equation}
\label{16}
\|f_{new}\|_{L^4(\R^3)}\lesssim_\epsilon N^{1/4}\delta^{-\epsilon}(\sum_B\|f_{new,B}\|^4_{L^4(\R^3)})^{1/4}.
\end{equation}

Let us call $\tau_{k,l}$ the boxes $L^{-1}_k(B).$ These boxes partition $\Nc_\delta(\theta_k)$ and are essentially flat. Each $\tau_{k,l}$ is essentially the $\delta$-neighborhood of some cap $\theta_{k,l}\subset \theta_k$. Figure 1 depicts the decomposition of some $\theta_2$ into four smaller caps $\theta_{2,l}$.
\medskip

Note that for each $B$
$$\widehat{f_{new,B}}=\widehat{f_{\tau_k,L^{-1}(B)}}\circ L_k^{-1}.$$
Thus, using a change of variables, \eqref{16} can be rewritten as
\begin{equation}
\label{13}
\|f_{\tau_k}\|_{L^4(\R^3)}\lesssim_\epsilon N^{1/4}\delta^{-\epsilon}(\sum_{\tau_{k,l}}\|f_{\tau_{k,l}}\|^4_{L^4(\R^3)})^{1/4}.
\end{equation}
The number $N$ is the same for each $\tau_k$. We can now define the partition $\Pc_\delta(S_k)$ to consist of all $\tau_{k,l}$ with $\tau_{k}\in\tilde{\Pc_\delta}(S_k)$. Combining \eqref{14} with \eqref{13} we get the following decoupling for a function $f$ with Fourier transform supported in $\Nc_\delta(S_k)$, $k\ge 1$
$$\|f\|_{L^4(\R^3)}\lesssim_{\epsilon}\delta^{-\epsilon}|{\Pc}_{\delta}(S_k)|^{\frac14}(\sum_{\tau_{k,j}\in {\Pc}_\delta(S_k)}\|f_{\tau_{k,j}}\|^4_{L^4(\R^3)})^{1/4}.$$
The partition $\Pc_\delta(S)$ will be the union of all $\Pc_\delta(S_k)$, $k\ge 0$.

\section{The perturbed cone}
\label{s4}

To simplify notation we will assume
\begin{equation}
\label{18}
\gamma(r)=r+(r-1)^n+O((r-1)^{n+1}).
\end{equation}
We will use the decomposition into intervals $U_k$ from the previous section
$$[1,1+\Delta)=[1,1+\delta^{1/n}]\cup \cup_{k\ge 1} [1+2^{k-1}\delta^{1/n},1+2^{k}\delta^{1/n}].$$
We continue to denote by $S_k$ the part of $S$ corresponding to $U_k$, and to write $s_k=2^k\delta^{1/n}$.
\medskip

Let us deal first with $S_0$. Note that $\Nc_\delta(S_0)$ sits inside $\Nc_{O(\delta)}(\Cc^2)$, so we can use the cone decoupling from Theorem \ref{3} to produce the relevant partition $\Pc_\delta(S_0)$, consisting of  essentially  flat boxes of dimensions $\sim \delta^{1/n}\times \delta^{1/2}\times \delta$.

Next, we fix some $k\ge 1$  and assume $f$ has  Fourier transform supported inside $\Nc_\delta(S_k)$. We will decouple in two stages. The first one is similar to the case $k=0$. More precisely, note that $\Nc_\delta(S_k)\subset \Nc_{O(s_k^n)}(\Cc^2)$. This allows us to run a cone  decoupling
\begin{equation}
\label{34}
\|f\|_{L^4(\R^3)}\lesssim_{\epsilon}\delta^{-\epsilon}|\tilde{\Pc}_{\delta}(S_k)|^{\frac14}(\sum_{\tau_k\in\tilde{\Pc}_\delta(S_k)}\|f_{\tau_k}\|^4_{L^4(\R^3)})^{1/4}.
\end{equation}
Each $\tau_k\in\tilde{\Pc}_\delta(S_k)$ is equal to  $\Nc_\delta(\theta_k)$ for some cap $\theta_k$ of dimensions $\sim s_k^{n/2}\times s_k$.
It is worth observing that cylindrical decoupling would have led to much wider caps of dimensions $\sim s_k^{1/2}\times s_k$. The caps $\theta_k$ however are small enough to behave well under rescaling.

We next perform a finer decoupling for each $\tau_k$. It is  convenient to deal with $\theta_k$ sitting directly above the $\xi_2$ axis, so that a point $(\xi_1,\xi_2,\xi_3)\in\theta_{k}$ satisfies
$$|\xi_1|\lesssim s_k^{n/2},\;\;\;\;\xi_2-1\sim s_k,\;\;\;\;\xi_3-1\sim s_k.$$
To understand better how to rescale $\theta_k$, we rotate it with $\frac\pi 4$ about the $\xi_1$ axis and rescale the $\xi_2$ and $\xi_3$ variables by $\sqrt{2}$. Thus, the new coordinates satisfy
$$\begin{cases}
	\xi_1&=\xi_1'\\ \xi_2&={\xi_2'-\xi_3'}\\ \xi_3&=\xi_2'+\xi_3'.
\end{cases}$$
Using \eqref{18}, the equation of $\theta_k$ appears now in an implicit form
$$\xi_2'+\xi_3'=\sqrt{\xi_1'^2+(\xi_2'-\xi_3')^2}+(\sqrt{\xi_1'^2+(\xi_2'-\xi_3')^2}-1)^n+O((\sqrt{\xi_1'^2+(\xi_2'-\xi_3')^2}-1)^{n+1}),$$
with
\begin{equation}
\label{21}
|\xi_1'|\lesssim s_k^{n/2},\;\;\;\xi_2'\sim 1+s_k,\;\;\;\xi_3'\sim s_k^n.
\end{equation}
This can be rearranged as follows
$$\xi_3'=\frac{\xi_1'^2}{4\xi_2'}+\frac1{4\xi_2'}(\sqrt{\xi_1'^2+(\xi_2'-\xi_3')^2}-1)^n+\frac1{4\xi_2'}O((\sqrt{\xi_1'^2+(\xi_2'-\xi_3')^2}-1)^{n+1}).$$
Note that $\xi_3'=\frac{\xi_1'^2}{4\xi_2'}$ is the equation of the cone $\Cc^2$ in the new coordinates.
\bigskip




We will use the transformation
$$L_k(\xi_1',\xi_2',\xi_3')=(\eta_1,\eta_2,\eta_3):=(\frac{\xi_1'}{s_k^{n/2}},\frac{\xi_2'-1}{s_k},\frac{\xi_3'-1}{s_k^n}).$$
Let us call $\theta_{k,new}=L_k(\theta_k)$. If we define $\xi_3'=\psi(\xi_1',\xi_2')$
then the equation of $\theta_{k,new}$ in the coordinates $\eta_1,\eta_2,\eta_3$ becomes
$$\eta_3=\psi_k(\eta_1,\eta_2):=\frac1{s_k^n}[\psi(s_k^{n/2}\eta_1,s_k\eta_2+1)-1],\;\;\;\;|\eta_1|,\;|\eta_2|\lesssim 1.$$
It remains to check that this surface satisfies the requirements in Theorem \ref{4}. More precisely, we have to show that the $C^3$ norm of $\psi_k$ is $O(1)$, independent of $k$. Also, we need to show that the Gaussian curvature is away from zero, uniformly over $k$.

\begin{lem}
Write
$$\phi(\xi_1',\xi_2')=\frac1{4\xi_2'}(\sqrt{\xi_1'^2+(\xi_2'-\psi(\xi_1',\xi_2'))^2}-1)^n+\frac1{4\xi_2'}O((\sqrt{\xi_1'^2+(\xi_2'-\psi(\xi_1',\xi_2'))^2}-1)^{n+1}).$$
Then for each $p,q\ge 0$ with
\begin{equation}
\label{20}
|\xi_1'|\lesssim s_k^{n/2},\;\;\;\;\xi_2'-1\sim s_k
\end{equation}
we have
$$|D_{1}^{p}D_2^{q}\phi(\xi_1',\xi_2')|\lesssim \min\{(s_k)^{n-p-q},1\}.$$
Consequently,  for each $p,q\ge 0$
$$\sup_{|\eta_1|,|\eta_2|\lesssim 1}|D_{1}^{p}D_2^{q}\psi_k(\eta_1,\eta_2)|\lesssim 1,$$
with an implicit constant independent of $k$.

\end{lem}
\begin{proof}
	It is clear that $\psi\in C^\infty$.
Note that due to \eqref{21} we have
$$\xi_2'\sim 1$$
and
$$|(\sqrt{\xi_1'^2+(\xi_2'-\psi(\xi_1',\xi_2'))^2}-1)^m|\lesssim s_k^m$$
for each $0\le m\le n$.
The bound on the derivatives of $\phi$ is now quite immediate, using repeated differentiation.

Next, recall that
$$\psi(\xi_1',\xi_2')=\varphi(\xi_1',\xi_2')+\phi(\xi_1',\xi_2'),$$
where
$$\varphi(\xi_1',\xi_2')=\frac{\xi_1'^2}{4\xi_2'}.$$
Note that in the domain \eqref{20} we have
$$|D_1^pD_2^q\varphi(\xi_1',\xi_2')|\lesssim (s_k^{n/2})^{2-p}$$
for each $0\le p\le 2$, $q\ge 0$
and the derivative becomes zero if $p\ge 3$.

Using all these observations, the desired bound on the derivatives of $\psi_k$ is now immediate.
\end{proof}
According to \eqref{22}, the Gaussian curvature of $\theta_{k,new}$ is roughly
$$Hess(\psi_k)=det\begin{bmatrix}\psi_{11}(s_k^{n/2}\eta_1,s_k\eta_2+1)
&\frac{(s_k)^{1+\frac{n}{2}}}{s_k^n}\psi_{12}(s_k^{n/2}\eta_1,s_k\eta_2+1)\\\frac{(s_k)^{1+\frac{n}{2}}}{s_k^n}\psi_{12}(s_k^{n/2}\eta_1,s_k\eta_2+1)&s_k^{2-n}\psi_{22}(s_k^{n/2}\eta_1,s_k\eta_2+1)
\end{bmatrix}.$$
It is immediate that $$Hess(\psi_k)=(s_k)^{2-n}Hess(\psi).$$
Another application of  \eqref{22} shows that $Hess(\psi)$ is roughly the Gaussian curvature of $\theta_k$, in the coordinates $(\xi_1',\xi_2',\xi_3')$. This is in turn comparable to the Gaussian curvature of $\theta_k$ in the original coordinates $(\xi_1,\xi_2,\xi_3)$. But \eqref{10} determines this curvature to be $\sim (s_k)^{n-2}$. We conclude that the curvature of $\theta_{k,new}$ is $\sim 1$, as desired.

\medskip

Note that $L_k$ maps $\Nc_\delta(\theta_k)$ inside $\Nc_{O(\frac{\delta}{s_k^n})}(\theta_{k,new})$. The rest of the argument is very similar to the one from the end of the previous section. We apply Theorem \ref{4} to partition $\Nc_{O(\frac{\delta}{s_k^n})}(\theta_{k,new})$ into essentially flat boxes with dimensions $\sim (\frac{\delta}{s_k^n})^{1/2}\times (\frac{\delta}{s_k^n})^{1/2}\times (\frac{\delta}{s_k^n}).$ Applying $L_k^{-1}$, this gives rise to a partition of $\Nc_\delta(\theta_k)$ into essentially flat boxes $\tau_{k,l}=\Nc_{\delta}(\theta_{k,l})$, where each $\theta_{k,l}$ has radial length $\sim s_k(\frac{\delta}{s_k^n})^{1/2}$ and angular length $\sim (s_k)^{n/2}(\frac{\delta}{s_k^n})^{1/2}$.

The desired partition  $\Pc_\delta(S_k)$ will consist of all boxes $\tau_{k,l}$ corresponding to all $\theta_k\in\tilde{\Pc}_\delta(S_k).$ This concludes the analysis of the perturbed cone.

\section{Final remarks}
\label{s5}

There are various ways in which one could refine the analysis in this paper. We have only aimed to prove a universal $l^4$ decoupling on the space $L^4$. A more careful inspection of the argument will reveal that sometimes  this can be naturally upgraded to an $l^2$ decoupling. For example, the torus \eqref{24} is positively curved on the outside ($r>1$) and negatively curved on the inside ($r<1$). Thus, the partition from Section \ref{3} leads in fact to an $l^2$ decoupling, while the analogous partition for the inside part leads only to an $l^4$ decoupling.

Also, the boxes in our partitions are maximal, subject to the requirement of being essentially flat. Under this mild  constraint  some surfaces perform better than  others. For example, we have seen earlier that the critical exponent for cone decoupling into plates is 6, rather than 4. Given a surface $S$, one may instead search for partitions consisting of boxes of smallest possible size, for which the $l^4$ decoupling holds. This issue seems to be much more delicate. For example,  one of the most interesting open  questions about surfaces in $\R^3$ is whether the cone can be decoupled into square-like caps. We conjecture the following.
 \begin{con}

 	 Let $\Pc_\delta(\Cc^2)$ be a partition of $\Nc_\delta(\Cc^2)$ into roughly $\delta^{-1}$ near rectangular boxes $\tau$ of dimensions $\sim \delta^{1/2}\times \delta^{1/2} \times \delta$. Then for each $f$ with Fourier transform supported in $\Nc_\delta(\Cc^2)$ and for $2\le p\le 4$ we have
 	$$
 \|f\|_{L^p(\R^3)}\lesssim_{\epsilon}(\delta^{-1})^{\frac12-\frac1p+\epsilon}(\sum_{\tau\in\Pc_\delta(\Cc^2)}\|f_\tau\|^p_{L^p(\R^3)})^{1/p}.
 $$

\end{con}
The only range where this is known to hold is $2\le p\le 3$,  using trilinear restriction technology.

\end{document}